\documentclass{article}

\usepackage{graphicx}
\usepackage{indentfirst}
\usepackage{amsmath,amsfonts,amsthm,amssymb}
\usepackage{mathrsfs,psfrag}
\usepackage{amscd}

\def\co{\colon\thinspace}

\DeclareMathAlphabet{\mathsfsl}{OT1}{cmss}{m}{sl}
\newcommand{\tensor}[1]{\mathsfsl{#1}}

\newtheorem{thm}{Theorem}[section]
\newtheorem{lem}[thm]{Lemma}
\newtheorem{cor}[thm]{Corollary}
\newtheorem{prop}[thm]{Proposition}
\newtheorem*{thm*}{Theorem}

\theoremstyle{definition}

\newtheorem{rem}[thm]{Remark}

\newcommand\sd{\mbox{-}}

\newcommand{\HF}{HF}

\newcommand{\Q}{\mathbb{Q}}
\newcommand{\R}{\mathbb{R}}

\newcommand{\Z}{\mathbb{Z}}
\newcommand{\E}{\mathbb{E}}

\newcommand{\cm}{\cdot}

\newcommand\SpinC{\mathrm{Spin}^c}
\newcommand{\F}{\mathbb F}

\newcommand\x{\mathbf x}

\newcommand\y{\mathbf y}

\newcommand\ModSphere{\ModFlow\left({\mathbb S}\longrightarrow 
\Sym^{g-1}(\Sigma_{1})\times \Sym^2(\Sigma_{2})\right)}
\newcommand\ModSpheres\ModSphere
\newcommand\CF{CF}

\newcommand\CFa{\widehat{CF}}

\newcommand\HFa{\widehat{HF}}

\newcommand\Mas{\mu}
\newcommand\UnparModSp{\widehat \ModSp}
\newcommand\UnparModFlow\UnparModSp
\newcommand\Mod\ModSp

\newcommand{\spinc}{\mathfrak s}

\newcommand\ModMaps{\mathcal M}
\newcommand\ModSp\ModMaps

\newcommand\Ta{{\mathbb T}_{\alpha}}
\newcommand\Tb{{\mathbb T}_{\beta}}

\newcommand\alphas{\mbox{\boldmath$\alpha$}}

\newcommand\betas{\mbox{\boldmath$\beta$}}

\newcommand\uCF{\underline\CF}
\newcommand\CFc{\CF^{\circ}}
\newcommand\uCFinf[1]{\uCF^\infty({#1})}
\newcommand\uCFp[1]{\uCF^+({#1})}
\newcommand\uCFm[1]{\uCF^-({#1})}
\newcommand\uCFa[1]{\underline{\CFa}({#1})}
\newcommand\uCFc[1]{\uCF^{\circ}({#1})}

\newcommand\uCFcR[1]{\uCF^{\circ}({#1};R_\omega)}

\newcommand\uCFcL[1]{\uCF^{\circ}({#1};\Lambda_\omega)}

\newcommand\CFcQ[1]{{\CFc}({#1};\Q)}

\newcommand\uHF{\underline\HF}
\newcommand\HFc{\HF^{\circ}}
\newcommand\uHFc[1]{\uHF^{\circ}({#1})}
\newcommand\uHFinf[1]{\uHF^\infty({#1})}

\newcommand\uHFaR[1]{\underline{\HFa}({#1};R_\omega)}

\newcommand\uHFpL[1]{\uHF^+({#1};\Lambda_\omega)}

\newcommand\uHFaL[1]{\underline{\HFa}({#1};\Lambda_\omega)}

\newcommand\HFaQ[1]{{\HFa}({#1};\Q)}

\newcommand\upartial{\underline{\partial}}

\newcommand\Dual{\mathcal D}
\newcommand\Duality\Dual

\newcommand\ons{Ozsv{\'a}th and Szab{\'o}}

\begin{document}

\title{Manifolds with small Heegaard Floer ranks}

\author{{Matthew HEDDEN\; and\; Yi NI}\\{\normalsize Department of Mathematics, MIT}\\
{\normalsize 77 Massachusetts Avenue, Cambridge, MA
02139-4307}\\{\small\it Emai\/l\/:\quad\rm mhedden@math.mit.edu}\\
{\small\rm yni@math.mit.edu}}

\date{}
\maketitle

\begin{abstract}
We show that the only irreducible three-manifold with positive
first Betti number and Heegaard Floer homology of rank two is
homeomorphic to zero-framed surgery on the trefoil.  We classify
links whose branched double cover gives rise to this manifold.
Together with a spectral sequence from Khovanov homology to the
Floer homology of the branched double cover, our results show that
Khovanov homology detects the unknot if and only if it detects the
two component unlink.
\end{abstract}

\section{Introduction}
In a sequence of papers, Ozsv{\'a}th and Szab{\'o} defined
invariants for a wide variety of topological and geometric objects
in low dimensions, including three- and four- manifolds, knots and
links, and contact structures
\cite{OSzAnn1,OSzAnn2,Knots,Links,Contact}.  These invariants
proved to be quite powerful, with striking applications to
questions in Dehn surgery, contact and symplectic geometry, knot
concordance, and questions about unknotting numbers (to name only
a few).

The Ozsv{\'a}th--Szab{\'o} invariants are particularly suited to
understand  homologically essential surfaces embedded in
three-manifolds.  In the context of knots, this is demonstrated by
the theorem that  knot Floer homology detects the Seifert genus of
a knot \cite{OSzGenus}.  More generally, the Floer invariants
capture the Thurston norm on the homology of link complements and
closed three-manifolds \cite{OSzGenus,OSzLinkTN}.  These theorems
have the immediate consequence that knot Floer homology detects
the unknot, as this is the only knot of Seifert genus zero.

In addition to the Thurston norm of a homology class, $\alpha\in
H_2(Y,\partial Y)$, the Ozsv{\'a}th--Szab{\'o} invariants answer
the subtler question of when a knot complement or closed
three-manifold fibers over the circle with fiber an embedded
surface, $\Sigma$, whose homology class equals $\alpha$
\cite{Gh,NiFibred,NiClosed,AiNi}. For knots, this again has a
beautiful corollary, namely that knot Floer homology detects the
trefoil and figure eight knots \cite{Gh}.  This follows from the
well-known fact that these are the only fibered knots with Seifert
genus one (together with an easy computation showing their Floer
homologies are different).

In this article we prove a similar theorem in the context of
closed three-manifolds.  Namely, we determine which irreducible
three-manifolds with positive first Betti number have rank $2$
Floer homology; there is only one.

\begin{thm}\label{thm:Rank2}
Suppose $Y$ is a closed irreducible $3$--manifold which satisfies
$b_1(Y)>0$. If $\mathrm{rank}\:\HFa(Y)=2$, then $Y$ is
homeomorphic (via a possibly orientation reversing map) to
$S^3_0(3_1)$, the manifold obtained by zero surgery on the
trefoil.
\end{thm}

Removing the irreducibility requirement, we obtain the following

\begin{cor}\label{cor:Rank2Gen}
Suppose $Y$ is a closed $3$--manifold with $b_1(Y)>0$. If the rank
of $\HFa(Y)$ is $2$, then $Y=Y'\#Z$, where $Y'$ is either
$\pm S^3_0(3_1)$ or $S^1\times S^2$ and $Z$ is an integer homology
sphere satisfying $\HFa(Z)\cong\mathbb Z$.
\end{cor}

In addition to our interest in the theorem as it pertains to
understanding the power of Heegaard Floer homology, we can use it
to gain new insight into the mysterious world of quantum
invariants.  Indeed, the existence of a spectral sequence from the
Khovanov homology of a link to the Heegaard Floer homology of its
branched double cover \cite{OSzBrCov} allows us to prove the
following theorem.  To state it, let $\mathbb F=\mathbb Z/2\mathbb
Z$.

\begin{thm}\label{thm:KhRank4}
Suppose $L\subset S^3$ is a link with $\det(L)=0$. Then
$$\mathrm{rank}_{\mathbb F} Kh(L;\mathbb F)>4$$ unless $L$ is
isotopic to a two-component split link $L=K_1\sqcup K_2$ satisfying
$\mathrm{rank}_{\mathbb F} Kh(K_i;\mathbb F)=2$, $i=1,2$.
\end{thm}

\begin{cor}\label{cor:Khcor}
Khovanov homology detects the unknot if and only if it detects the two-component unlink.
\end{cor}

This theorem should be compared with the results of
\cite{Thistlethwaite,EKT} which provide examples of two (or more)
component links with Jones polynomial equal to that of the unlink.
By construction these links are non-split, so our result implies
that they are distinguished from the unlink by their Khovanov
homology.  This highlights the strength of Khovanov homology over
its Euler characteristic, and indicates that the question of
whether Khovanov homology detects the unknot may be of a much
different nature than the corresponding question for the Jones
polynomial.

The proof of Theorem~\ref{thm:Rank2} begins by using the Thurston
norm detection of Floer homology  to show the rank assumption
implies that the manifold contains a homologically essential
torus.  The ability of Floer homology to detect fibering then shows
that this torus is the fiber in a surface bundle over the circle.
Lower bounds on the rank of Floer homology in terms of
$H_1(Y;\Z)$, together with a simple analysis of this group for
torus fibrations, shows that the manifold is zero surgery on the
trefoil  or figure eight.  The latter manifold, however, has Floer
homology of rank four by direct calculation.

Theorem~\ref{thm:KhRank4} follows from the spectral sequence from
Khovanov homology to the Floer homology of the branched double
cover mentioned above.  The main work in this step is to
understand which links in the three-sphere have $S^3_0(3_1)$
as their branched double cover.  We classify such links (there are
essentially only two) in Section \ref{sec:whichlinks} by a
geometric argument based on the fact that $S^3_0(3_1)$ admits a
Euclidean geometry.

We conclude by remarking that the question of whether Khovanov
homology detects the unknot could likely be understood through
Heegaard Floer homology if the manifolds, $Z$, appearing in
Corollary~\ref{cor:Rank2Gen} could be appropriately classified.
Indeed, the question of which integer homology three-spheres have
rank one Floer homology (the so-called homology sphere L-spaces)
is interesting for a variety of reasons.  Even with our limited
understanding of homology sphere L-spaces, one can gain useful
information about Khovanov homology.  For instance, one can easily
show that the Khovanov homology of the two-cable (and many other
satellite operations) detects the unknot \cite{2cable,
TangleUnknotting}. See also, \cite{Ef,GW}, for related results.

\bigskip

\noindent{\bf Acknowledgements.}\quad We are very grateful to
Peter Ozsv\'ath for a helpful conversation about the Universal
Coefficients Theorem. We thank Yanki Lekili for helpful
conversations. We also thank Chuck Livingston and Alan Edmonds for
their interest in this work. The first author was partially
supported by NSF grant number DMS-0706979, and thanks the Indiana
University math department for their hospitality during an
extended visit in which some of this work was completed. The
second author is partially supported by an AIM Five-Year
Fellowship and NSF grant number DMS-0805807.

\section{Preliminaries}

In this section we recall some necessary background on the
Heegaard Floer  and Khovanov homology theories, respectively. Our
purpose is mainly to establish notation and collect results which
will be used in the subsequent sections.  For the unfamiliar
reader, we refer to \cite{FloerClay2,OSzSurvey} for an
introduction to Ozsv{\'a}th--Szab{\'o} theory and
\cite{Kh,BarNatan} for material on Khovanov homology.

\subsection{Twisted Heegaard Floer homology}\label{subsec:twisted}
Heegaard Floer homology (or Ozsv{\'a}th--Szab{\'o} homology)
assigns chain complexes to a $\SpinC$ three-manifold,
$(Y,\spinc)$.  One can import the  general construction of
homology with twisted coefficients (see, for instance
\cite{Hatcher}) into this theory in a variety of useful ways,
which we now recall.

The input for the theory is an admissible pointed Heegaard diagram
$$(\Sigma,\alphas,\betas,z)$$ for $(Y,\spinc)$.  Taking the
symmetric product of the diagram, one arrives at the
$2g$-dimensional symplectic manifold Sym$^g(\Sigma)$, together
with two Lagrangian submanifolds $\Ta,\Tb,$ and a distinguished
hypersurface, $V_z$, which arise from the attaching curves and
basepoint, respectively (see \cite{OSzAnn1} for more details,
specifically Sections $2$ and $4$).

The most general construction of Heegaard Floer homology with
twisted coefficients defines a chain complex, $\uCFinf{Y,\spinc}$,
which is freely generated over the ring
$\Z[U,U^{-1}]\otimes_\Z\Z[H^1(Y;\Z)]$ by intersection points $\x
\in\Ta\cap\Tb$ whose  associated $\SpinC$ structure corresponds to
$\spinc$. Here $H^1(Y;\Z)$ is the first singular cohomology of
$Y$, and $U$ is a formal variable of degree $-2$.

Given $\x,\y\in \Ta\cap\Tb$, let  $\pi_2(\x,\y)$ denote the set of
homotopy classes of Whitney disks connecting $\x$ to $\y$.  The
twisted coefficient ring is a reflection of the fact that
$\pi_2(\x,\y)\cong\Z\oplus H^1(Y;\Z)$.\footnote{To identify this
with the standard construction of homology with twisted
coefficients, observe that $\pi_2(\x,\y)$ is the fundamental group
of the configuration space of paths in Sym$^g(\Sigma)$ from $\Ta$
to $\Tb$.  Heuristically, Heegaard Floer homology is the Morse
homology of this configuration space with respect to a specific
action functional.}  Pick an additive assignment  $A:
\pi_2(\x,\y)\rightarrow H^1(Y)$ in the sense of Definition $2.12$
of \cite{OSzAnn1}, and define an endomorphism
$$ \upartial:  \uCFinf{Y,\spinc} \rightarrow \uCFinf{Y,\spinc},$$
by the formula:
$$\upartial \x=\sum_{\y\in\Ta\cap\Tb}
\sum_{\{\phi\in\pi_2(\x,\y)\big|\Mas(\phi)=1\}}
\#\left(\UnparModFlow(\phi)\right)U^{n_z(\phi)}\otimes
e^{A(\phi)}\cm  \y,$$ where $\UnparModFlow(\phi)$ denotes the
quotient of the moduli space of unparametrized $J$-holomorphic
disks representing the homotopy class $\phi$, $\mu(\phi)$ is the
Maslov index, and $n _z(\phi)=\#\phi\cap V_z$ is the algebraic
intersection of $\phi$ with the hypersurface. Here, we have
denoted elements in $\Z[H^1(Y)]$ in the typical exponential
notation.  Gromov compactness for $J$-holomorphic curves ensures
that $\upartial\circ\upartial=0$, and we denote the homology of the
resulting chain complex by $\uHFinf{Y,\spinc}$. The results of
\cite{OSzAnn1,OSzAnn2} indicate that these groups depend only on
the $\SpinC$-diffeomorphism type of the pair $(Y,\spinc)$.  We
refer the reader to the aforementioned papers for more details.
Section $8$ of \cite{OSzAnn2}, in particular, introduces this
notion of twisted coefficients.

We could alternatively take the chain complexes to be generated
over either of the rings  $$U\cm\Z[U]\otimes H^1(Y) \ \ \ \ \ \ \
\ \ \ \Z[U,U^{-1}]/U\cm\Z[U]\otimes H^1(Y).$$  We denote the
resulting complexes by $\uCFm{Y,\spinc}, \uCFp{Y,\spinc}$,
respectively.  Positivity of intersections for $J$-holomorphic
curves ensures that this is well-defined and, moreover, that we
have a short exact sequence:
$$\begin{CD}
0\to \uCFm{Y,\spinc}\to \uCFinf{Y,\spinc}\to \uCFp{Y,\spinc}\to0
\end{CD}
$$
(with corresponding long exact sequence of homology). We can also
define a complex $\uCFa{Y,\spinc}$ by the following exact sequence
$$\begin{CD}
0\to \uCFa{Y,\spinc}\to \uCFp{Y,\spinc}\overset{\cm U}\to \uCFp{Y,\spinc}\to0
\end{CD}
$$

The homology of the various complexes are invariants of the pair
$(Y,\spinc)$ and are denoted $\uHF^+,\uHF^-,\widehat{\uHF}$.

Note that  $\Z[U,U^{-1}]\otimes \Z[H^1(Y)]$ is naturally a
$\Z[H^1(Y)]$ module (by letting $\Z[H^1(Y)]$ act trivially on
$\Z[U,U^{-1}]$).  Thus, for any $\Z[H^1(Y)]$-module, $M$, it makes
sense to consider Heegaard Floer homology with coefficients in
$M$.  By definition, these are the groups:
$$\uHFc{Y,\spinc;M}:=H_*(\uCFc{Y,\spinc}\otimes_{\Z[H^1(Y)]} M),$$
where $\uCF^\circ$ denotes any of the chain complexes ($\widehat{\pm},\infty$)  considered above.

We will be interested in $\Z[H^1(Y)]$-modules which focus
attention on the part of $H^1(Y)$ which pairs non-trivially
with certain $1$-dimensional homology classes.  More precisely,
given $\omega\in H_1(Y)$, we consider the ring $R=\Q[T,T^{-1}]$ as
a $\Z[H^1(Y)]$ module by defining the action:
$$ e^{\gamma}\cm 1 =  T^{\gamma(\omega)},$$
where $1\in \Q[T,T^{-1}]$, $e^\gamma\in \Z[H^1(Y)]$, and
$\gamma(\omega)$ is the natural pairing between cohomology and
homology.  We denote chain complexes with coefficients in $R$ by
$$\uCFcR{Y,\spinc}:= \uCFc{Y,\spinc}\otimes_{\Z[H^1(Y)]} R,$$
and refer to them as the $\omega$-twisted Heegaard Floer chain
complexes. We can also complete this coefficient ring in a particularly
useful way.  Define an $R$-module by
$$\Lambda=\left\{\left. \sum_{r\in\R} a_r T^r \right|  a_r\in \R, \#\{a_r| a_r \ne 0, r\le c\}<\infty \ \ \text{for \ any \ } c\in \R \right\},$$
where $R$ acts on $\Lambda$ by polynomial multiplication. We refer
to $\Lambda$ as the {\em universal Novikov ring}, and the
corresponding chain complexes
$$\uCFcL{Y,\spinc}:= \uCFcR{Y,\spinc}\otimes_R \Lambda$$
as Heegaard Floer chain complexes with $\omega$-twisted Novikov
coefficients or, following \cite{JM}, as the $\omega$-{\em
perturbed} Floer homology. Observe that $\Lambda$ is a field.

Note that we can disregard the twisting altogether by letting $R$
act trivially on $\Q$ (or any other ring, e.g $\Z,$ or $\Z/2\Z$)
by the rule $T\cm a = a,$ where $a\in \Q$ (with corresponding
trivial action on any other ring). The resulting chain complexes
 $$ \CFcQ{Y,\spinc}:=\uCFcR{Y,\spinc}\otimes_R \Q$$
are the ordinary (untwisted) Heegaard Floer chain complexes with
coefficients in $\Q$.   Throughout, we will adopt the notation
that $\HFc(Y,\spinc):= \HFc(Y,\spinc;\Z)$, i.e. in the untwisted
world we use $\Z$ coefficients unless otherwise specified.

Finally, given $[F]\in H_2(Y;\Z)$, we can consider the direct sum
of chain complexes corresponding to $\SpinC$ structures on $Y$
whose first Chern class evaluates on $[F]$ as a specified integer.
In this case, we adopt the notation:
$$  \CFc(Y,[F],i):= \underset{\{\spinc\in \SpinC(Y)|  \langle c_1(\spinc),[F]\rangle=2i\} }\bigoplus \CFc(Y,\spinc),$$
with corresponding notation for homology groups and the various coefficient rings described above.

\subsection{Non-triviality theorems}

In this subsection, we collect a few key theorems concerning
non-triviality of Floer homology.  We first state  a
generalization of \ons's theorem that Floer homology detects the
Thurston norm of a closed three manifold.  To do this recall that
the {\em complexity} of a closed surface, $F$, is the quantity:
$$x(F)=\sum_{F_i\subset F} \text{max}\{0,-\chi(F_i)\},$$
where the sum is taken over all connected components of $F$ and
$\chi$ is the Euler characteristic. A homologically non-trivial
surface, $F$, in an irreducible three-manifold is called {\em
taut} if it minimizes complexity amongst all embedded surfaces
whose associated homology class is equal to $[F]\in H_2(Y;\Z)$. We
have the following
\begin{thm}{\rm\cite[Theorem~3.6]{NiNSSphere}}\label{thm:TwistNorm}
Suppose $Y$ is a closed irreducible $3$--manifold, and $F$ is a taut
surface in $Y$. Then there exists a nonempty open set $U\subset
H_1(Y;\mathbb R)$, such that for any $\omega\in U$,
$$\uHFpL{Y,[F],\frac12x(F)}\ne0, \quad \uHFaL{Y,[F],\frac12x(F)}\ne0.$$
\end{thm}

This implies the following for the homology with untwisted coefficients  following \begin{thm}\label{thm:UnTwistNorm}
Suppose $Y$ is a closed $3$--manifold, $F$ is a taut surface in
$Y$. Then
$${\HFa}(Y,[F],\frac12x(F))\otimes\mathbb Q\ne0.$$
\end{thm}
\begin{proof}
Since $R=\Q[T,T^{-1}]$ is a PID, the universal coefficients
theorem \cite{Sp} implies
$$\uHFaL{-}\cong \uHFaR{-}\otimes_R \Lambda \bigoplus \mathrm{Tor}_{R}(\uHFaR{-},\Lambda).$$
Since $\Lambda$ is $R$ torsion-free the Tor term vanishes. The
non-vanishing result for the left-hand side (Theorem
\ref{thm:TwistNorm}) then implies that
$\uHFaR{Y,[F],\frac12x(F)}\otimes_R \Lambda$, and hence
$\uHFaR{Y,[F],\frac12x(F)}$, is non-zero. Applying the universal
coefficients theorem again,
$$\HFaQ{-}\cong \uHFaR{-}\otimes_R \Q \bigoplus \mathrm{Tor}_{R}(\uHFaR{-},\Q),$$ we find that the untwisted Floer homology with $\Q$ coefficients is non-trivial.  On the other hand, this latter group is isomorphic to
$$\HFaQ{-}\cong \HFa(-)\otimes_{\Z} \Q \bigoplus \mathrm{Tor}_{\Z}(\HFa(-),\Q).$$
The Tor term vanishes since $\Q$ has no $\Z$ torsion.  This yields the result.
\end{proof}

\begin{lem}\label{lem:HFinfty}
If $HF^{\infty}(Y,\mathfrak s)$ contains $k$ copies of $\mathbb
Z[U,U^{-1}]$ as direct summands, then $$\mathrm{rank}\:\HFa(Y,\mathfrak s)\ge
k.$$
\end{lem}
\begin{proof}
This follows easily from the two exact sequences
$$\begin{CD}
\cdots\to {HF^-}\to {HF^{\infty}}\to{{HF^+}}\to\cdots,\\
\cdots\to {HF^+}\to{HF^+}\to{\HFa}\to\cdots,
\end{CD}
$$
and the fact that $HF^-_i=0$ when $i$ is sufficiently large.
\end{proof}

\subsection{Khovanov homology}
To a link $L\subset S^3$, Khovanov homology associates a
collection of bigraded abelian groups,  $Kh_{i,j}(L)$ \cite{Kh}.
The graded Euler characteristic of these groups is the Jones
polynomial, in the sense that $$(q+q^{-1})\cm J_L(q^2)= \sum_j
(\sum_i (-1)^i \mathrm{rank}\:Kh_{i,j}(L)) \cm q^j,$$ where
$J_L(q)$ is the Jones polynomial, and ``rank" is taken to mean the
rank as a $\Z$-module.  For the present purpose, it will not be
necessary to explain the precise details of Khovanov's
construction.  It suffices to say that the groups arise as the
(co)homology groups of a bigraded (co)chain complex,
$CKh_{i,j}(D)$, associated to a link diagram, $D$.  The complex is
obtained by applying a $(1+1)$ dimensional topological quantum
field theory to the cube of complete resolutions of $D$. The
$i$-grading is the cohomological grading i.e. the differential
increases this grading by one, while the $j$-grading is the
so-called quantum grading (corresponding to the variable in the
Jones polynomial).  In the original treatment, the differential
preserved the $j$-grading.  Lee considered a perturbation of this
differential which does not preserve the $j$-grading, but instead
makes $CKh_{i,j}(L)$ into a complex filtered by $j$ \cite{Lee}. This perturbation was useful for several purposes, most notably in
Rasmussen's combinatorial proof of Milnor's conjecture on the
unknotting number of torus knots \cite{RasSlice}.  A consequence
of Lee's work which will be useful for us is the
following
\begin{thm}{\rm\cite{Lee}}\label{thm:Lee}
Let $L\subset S^3$ be a link of $|L|$ components. Then $$\mathrm{rank}\: Kh(L)\ge 2^{|L|}.$$
\end{thm}

The main feature of Khovanov homology which we use is a connection with the Heegaard Floer invariants.

\begin{thm} {\rm\cite[Theorem~1.1]{OSzBrCov}}\label{thm:Branched}
Let $L \subset S^3$  be a link. There is a spectral sequence whose
$E_2$ term consists of $\widetilde{Kh}(\overline{L};\Z/2\Z)$,
Khovanov's reduced homology of the mirror of $L$, and which
converges to  $\HFa(\Sigma(L);\Z/2\Z)$, the Heegaard Floer
homology of the branched double cover of $L$.  \end{thm} The
reduced Khovanov homology is a variant of Khovanov homology
defined via a chain complex, $\widetilde{CKh}$ which has half the
rank (taken in any coefficient ring) of $CKh(L)$.  In the case of
$\Z/2\Z$ coefficients, there is little difference between the
reduced and ordinary theories.  Indeed, \begin{equation}
\label{eq:Red} \mathrm{Kh}(L;\Z/2\Z) \cong
\widetilde{\mathrm{Kh}}(L;\Z/2\Z)\otimes V,\end{equation} where
$V$ is the rank $2$ vector space over $\Z/2\Z$ obtained as the
Khovanov homology of the unknot (see, for instance,
\cite{Shumakovitch2004}).

Packaging all of this, the result we use is the following:
\begin{prop}\label{prop:ssbound}
Let $L\subset S^3$ be a link of $|L|$ components and $\Sigma(L)$ denote its branched double cover.
$$\mathrm{rank}_{\Z/2\Z}\: Kh(L;\Z/2\Z)\ge 2\cm \mathrm{rank}_\Z\: \HFa(\Sigma(L))$$
\end{prop}
\begin{proof}   Theorem \ref{thm:Branched} immediately yields
$$ \mathrm{rank}_{\Z/2\Z} \widetilde{Kh}(\overline{L};\Z/2\Z) \ge \mathrm{rank}_{\Z/2\Z} \HFa(\Sigma(L);\Z/2\Z).$$
The left-hand side of this inequality is equal to
$\mathrm{rank}_{\Z/2\Z} \widetilde{Kh}({L};\Z/2\Z)$ by a duality
theorem for Khovanov homology under
taking mirror images \cite[Corollary 11]{Kh}. We have $$\mathrm{rank}_{\Z/2\Z}
\HFa(\Sigma(L);\Z/2\Z)\ge \mathrm{rank}_{\Z} \HFa(\Sigma(L)),$$ by
the universal coefficients theorem. Multiplying by $2$ and
combining with Equation \eqref{eq:Red} yields the desired
inequality.
\end{proof}

\section{Manifolds with rank $2$ Heegaard Floer homology}
In this section we prove our theorem characterizing manifolds with
rank $2$ Floer homology.  The basic idea of the proof is simple.
We first use the rank assumption to show that the only non-trivial
class in $H_2(Y;\Z)$ is represented by a torus.  The rank being this
small further implies, by work of Ai and the second author
\cite{AiNi}, that this torus is a fiber in a fibration of the
three-manifold over the circle.  The scarcity of three-manifolds
which fiber in this way then pins down the manifold exactly.  We
begin.

\begin{proof}[Proof of Theorem~\ref{thm:Rank2}]Suppose $Y$ is an irreducible closed  three-manifold with $b_1(Y)>0$ and
$\mathrm{rank}\:\HFa(Y)=2$.

\vspace{5pt}\noindent{\bf Claim 1.} {\sl$Y$ contains a
non-separating torus.}\vspace{5pt}

Since $b_1(Y)>0$, there exists a closed connected surface
$F\subset Y$, such that $[F]\ne0\in H_2(M)$ and $F$ is Thurston
norm minimizing. $F$ is not a sphere since $Y$ is irreducible so,
in particular, $F$ is taut. By Theorem~\ref{thm:UnTwistNorm},
$\HFa(Y,[F],g-1)\ne0$, where $g$ is the genus of $F$. A symmetry
property of Floer homology further implies that
$\HFa(Y,[F],1-g)\ne0$ \cite[Theorem~2.4]{OSzAnn2}. Note that since
$b_1(Y)>0$, \cite[Proposition~5.1]{OSzAnn2} implies
$\chi(\HFa(Y,\mathfrak s))=0$. In particular $\HFa(Y,[F],g-1)$ and
$\HFa(Y,[F],1-g)$ each have rank at least $2$. If $g>1$, then
$$\mathrm{rank}\:\HFa(Y)\ge\mathrm{rank}\:\HFa(Y,[F],g-1)+\mathrm{rank}\:\HFa(Y,[F],1-g)\ge4,$$
which is impossible. So $F$ must be a torus.

\vspace{5pt}\noindent{\bf Claim 2.} {\sl $Y$ is a
 torus bundle over the circle.}\vspace{5pt}

The tool for proving this claim will be the main theorem of
\cite{AiNi}, which indicates that Floer homology detects torus
fibrations:
\begin{thm}{\rm \cite[Theorem~1.2]{AiNi} }Suppose $F\subset Y$ is an embedded torus and there exists $\omega\in H_1(Y;\Z)$ satisfying $\omega\cm [F]\ne 0$, and for which
$$ \mathrm{rank}_\Lambda \uHFpL{Y}=1.$$
Then $Y$ fibers over the circle with fiber $F$.
\end{thm}

Note that to be consistent, we have stated this theorem in terms
of coefficients twisted by  $\omega\in H_1(Y;\Z)$ rather than
$\omega\in H^2(Y;\R)$ as in \cite{AiNi}.  The definition of the
modules involved, together with Poincar{\'e} duality, shows that
the chain complexes are isomorphic.  Also recall that $\uHFaL{Y}$
denotes the sum of Floer groups over all $\SpinC$ structures.

We wish to apply the above theorem.  By
Theorem~\ref{thm:TwistNorm}, we can choose a homology class
$\omega\in H_1(Y)$ with $\omega\cdot[F]\ne0$ and
$\uHFpL{Y,[F],0}\ne0$.  Moreover, an application of the adjunction
inequality \cite[Theorem~7.1]{OSzAnn2} shows that
$\uHFpL{Y,[F],0}=\uHFpL{Y}$. More precisely, the adjunction
inequality adapted to twisted coefficients tells us that
$\uHFpL{Y,\spinc}=0$ for any $\SpinC$ structure satisfying
$\langle c_1(\spinc), [F]\rangle \ne 0$. Thus it remains to show
that our assumption $\mathrm{rank}_\Z \HFa(Y)=2$ implies that
$\mathrm{rank}_\Lambda \:\uHFpL{Y,[F],0}=1$.  This will follow
easily from the universal coefficients theorem and the exact
sequence relating $\uHF^+$ to $\widehat{\uHF}$.

As an intermediary, let us consider the $\omega$-twisted Heegaard
Floer homology $\uHFaR{Y,[F],0}$ (recall that $R=\Q[T,T^{-1}]$).
As in Subsection \ref{subsec:twisted} we have two natural
$R$-modules: the trivial module $\Q$ and the universal Novikov
ring $\Lambda$. Correspondingly, we have the untwisted Heegaard
Floer homology $\HFaQ{Y,[F],0}$ and the $\omega$-perturbed Floer
homology $\uHFaL{Y,[F],0}$.

We can apply the universal coefficients
theorem to get
$$\HFaQ{Y,[F],0}\cong \uHFaR{Y,[F],0}\otimes_{R}\mathbb Q\bigoplus
\mathrm{Tor}_{R}(\uHFaR{Y,[F],0},\mathbb Q).$$ By assumption, we have
$\HFaQ{Y,[F],0}\cong\mathbb Q^2$, so
$$\mathrm{rank}_R\uHFaR{Y,[F],0}\le2.$$

Again by the universal coefficients theorem, we have
$$\uHFaL{Y,[F],0}\cong \uHFaR{Y,[F],0}\otimes_{R}\Lambda\bigoplus
\mathrm{Tor}_{R}(\uHFaR{Y,[F],0},\Lambda).$$ Since
$\Lambda$ is $R$-torsion free, we have
$\mathrm{Tor}_{R}(\uHFaR{Y,[F],0},\Lambda)=0$. Thus
$$\mathrm{rank}_{\Lambda}\uHFaL{Y,[F],0}\le2.$$

Having bounded the rank of $\underline{\HFa}$, recall the exact sequence
$$\begin{CD}
\cdots\to { \underline{HF}^+}@>U>> {\underline{HF}^+}\to {\underline{\HFa}}\to\cdots.
\end{CD}
$$
The techniques of \cite{Lek} show that $U=0$. Since $\uHFpL{Y,[F],0}\ne0$, we must have
$\mathrm{rank}_{\Lambda}\uHFpL{Y,[F],0}=1$, which completes the proof of the claim.

\vspace{5pt}\noindent{{\bf Claim 3.} $Y$ is obtained by zero
surgery on the trefoil knot.}\vspace{5pt}

Since $Y$ is a torus bundle, $b_1(Y)\le3$. If $b_1(Y)=3$, then
$Y=T^3$, for which $\HFa$ is known to be isomorphic to $\mathbb Z^6$
\cite[Proposition~8.4]{OSzAbGr}. If $b_1(Y)=2$ or $H_1(Y;\mathbb
Z)$ contains torsion, \cite[Theorem~10.1]{OSzAnn2} and
Lemma~\ref{lem:HFinfty} imply that
$\mathrm{rank}\:\HFa(Y)\ge4$. Thus
$H_1(Y;\mathbb Z)\cong\mathbb Z$.

Let $\tensor A\in SL(2,\mathbb Z)$ be the matrix representing the
monodromy of the torus bundle. In order to have $H_1(Y;\mathbb Z)\cong\mathbb
Z$, we must have $\det(\tensor A-\tensor I)=\pm1$ (by Mayer--Vietoris), so
$\mathrm{trace}\:\tensor A=1\;\text{or}\;3$. Up to conjugacy in
$SL(2,\mathbb Z)$, there are only three such matrices
\cite[21.15]{Z}, which corresponding to the zero surgeries on the
two trefoil knots and on the figure-8 knot. The first two
manifolds have $\HFa\cong\mathbb Z^2$ while the last one
has $\HFa\cong\mathbb Z^4$ \cite[Section~8]{OSzAbGr}.
\end{proof}

\section{The zero surgery on the trefoil as a double branched cover}
\label{sec:whichlinks}

For a link $L\subset S^3$, we let $\Sigma(L)$ denote the branched
double cover of $S^3$, branched along $L$. We will denote the
manifold obtained by zero surgery on the trefoil by $M$.  In this
section, we classify the links for which $\Sigma(L)\cong M$.
It turns out there are only two.

\begin{prop}\label{prop:TreDoub}
Suppose $\Sigma(L)\cong M,$ the manifold obtained by $0$--surgery
on the trefoil knot. Then $L$ is isotopic to one of the two links,
$H_{1,3}$ or $H_{\sd 1,\sd 3}$, pictured in
Figure~\ref{fig:hopfcable}.
\end{prop}
\begin{rem} We consider $M$ as an unoriented manifold, and hence the links are specified up to taking mirror images.
\end{rem}

We will prove the proposition by a geometric argument, which we now
sketch. The argument begins with the observation that $M$ admits a
Euclidean geometric structure.  Existing results about group
actions on such manifolds will show that the involution presenting
$M$ as a branched double cover can be assumed to be an isometry.
Analyzing this situation, we will show that the involution
restricts to an involution on the two pieces of the decomposition
of $M$ as a torus bundle $$M\cong T^2\times I \underset{\tensor A}\cup
T^2\times I.$$ Using this fact, we can show that $L$ must result
from $2$-cabling each component of the Hopf link, as shown in
Figure \ref{fig:hopfcable}.  An analysis of $H_1(\Sigma(L);\Z)$
for such $L$ then specifies the link exactly. With the general
idea in place, we begin.

\begin{figure}
 \psfrag{n}{$n$}
 \psfrag{m}{$m$}
 \psfrag{1}{$\ 1$}
  \psfrag{-1}{$\ \sd 1$}
\psfrag{=}{$\sim$} \psfrag{H(n,m)}{$H_{n,m}$}
\begin{center}
\includegraphics[width=250pt]{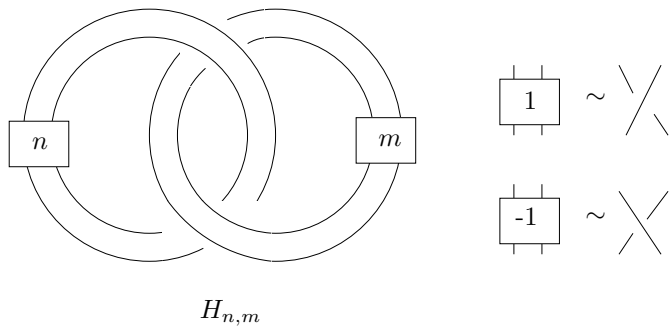}
\caption{\label{fig:hopfcable}A satellite of the Hopf link formed
by $2$-cabling each of its components. The numbers  $n$ and $m$
indicate the number of  crossings, according to the convention on
the right.  }
\end{center}
\end{figure}

\begin{proof}[Proof of Proposition~\ref{prop:TreDoub}]

Suppose $M$ is homeomorphic to $\Sigma(L)$. Let $\pi\co M\to S^3$
be the branched covering map, and let $\rho\co M\to M$ be the deck
transformation, which is an involution. Then $\pi^{-1}(L)$ is the
set of fixed points of $\rho$.  The following lemma captures the
geometry of our setup.

\begin{lem}
There exists a flat metric on $M$ which is preserved by $\rho$.
With respect to this metric, there exists a totally geodesic
(embedded) torus $T$ representing a generator of $H_2(M)\cong \Z$.
\end{lem}
\begin{proof}
$M$ is a torus bundle over $S^1$ with a $6$--periodic monodromy,
so the $6$--fold cyclic cover of $M$ is $T^3$.  Thus $M$ admits a
Euclidean structure. According to \cite[Theorem~2.1]{MS}, there
exists a flat metric on $M$ with respect to which $\rho$ is an isometry.

Suppose $f\co T^2\to M$ is a least area map representing a
generator of $H_2(M)$.   \cite[Theorem~5.1]{FHS} implies that
either (i) $f$ is an embedding, or (ii) $f$ double covers an
embedded one-sided surface in $M$.  The latter is ruled out  by
the fact that $f_*([T^2])\ne 0\in H_2(M)$, hence $f$ is an
embedding  (strictly speaking, to apply the theorem we must also
verify that $M$ is irreducible and contains no $\R P^2$, but this
follows easily from the fact that $M$ is a torus bundle).   Let
$T$ denote the image of $f$.  Since $T$ is area minimizing, it is
a minimal surface, and hence the mean curvature vector vanishes
identically. Since $M$ is flat,  the Gaussian curvature  of $T$ is
the product of its principal curvatures, and is non-positive since
the mean curvature (the average of the principal curvatures)
vanishes. The Gauss--Bonnet Theorem then implies that the Gaussian
curvature of $T$ is $0$ everywhere, and hence the principal
curvatures also vanish everywhere, so $T$ is totally geodesic.
\end{proof}

The geometry at hand tightly constrains the behavior of the
covering involution with regard to the torus.  We have the
following

\begin{lem}
In the previous lemma, one can assume that the totally geodesic torus satisfies $\rho(T)\cap T=\emptyset$
and $[\rho(T)]=-[T]\in H_2(M)$.
\end{lem}
\begin{proof}
Since $H_2(M)\cong\mathbb Z$, $\rho(T)$ is homologous to $\pm T$.
Endow the three-torus, $T^3$, with the flat metric obtained by
pulling back the metric on $M$ under the $6$--fold cyclic covering
map  $p\co T^3\to M$. We will compare the lifts of $T$ and
$\rho(T)$ under $p$.

Let $\widetilde T$ denote a component of $p^{-1}(T)$.  $\widetilde
T$ is, like $T$,  totally geodesic.  The surface
$p^{-1}\circ\rho(T)$ is a disjoint union of totally geodesic tori
in $T^3$, such that each component is homologous to
$\pm[\widetilde T]$. We claim that $\widetilde T$ is either a
component of $p^{-1}\circ\rho(T)$ or disjoint from
$p^{-1}\circ\rho(T)$. This follows from elementary Euclidean
geometry.  Indeed, since all the tori are geodesic they lift to
planes in the universal cover, $\E^3$.   The assumption that the
tori are homologous up to sign implies that these planes are
parallel or have a pair of components which coincide (if not, they
would intersect in a line that projects to a non-separating circle
of intersection between the tori in $T^3$, contradicting the
homological assumption).   This proves the claim.  In the case
that $\widetilde T$ is  disjoint from $p^{-1}\circ\rho(T)$ we see
that $\rho(T)\cap T=\emptyset$, as desired.  If  $\widetilde T$ is
a component of $p^{-1}\circ\rho(T)$, then $\rho(T)=T$ as a set.

In the case that $\rho(T)=T$, we claim that $\rho_*[T]=-[T]$.
Suppose, otherwise, that $\rho_*[T]=[T]$. Then $\rho|_T$ is an
orientation preserving involution. Thus $\pi(T)$ is a closed
oriented surface in $S^3$. Choose a simple closed curve
$\gamma\subset M$ intersecting $T$ exactly once. Then
$\pi(\gamma)$ is a closed curve in $S^3$ which has algebraic
intersection number $2$ with $\pi(T)$, contradicting the fact that
$H_1(S^3)=0$.

Thus $\rho_*[T]=-[T]$ if $\rho(T)=T$. For some small
$\varepsilon>0$, consider the set
$$N_\varepsilon=\{x\in M|\:\mathrm{dist}(x,T)\le \varepsilon\}$$
Then $\partial N_\varepsilon$ consists of two flat tori,
$T_\varepsilon$ and $T_{\sd \varepsilon}$. The assumptions that
$\rho_*[T]=-[T]$ and $\rho(T)=T$ imply that
$\rho(T_\varepsilon)=T_{-\varepsilon}$. By working with
$T_{\varepsilon}$ instead of $T$,  we obtain the desired
conclusion.

Finally, if $\rho(T)\cap T=\emptyset$, we claim that $\rho_*[T]=-[T]$
 also holds. Indeed, $T$ and $\rho(T)$ split $M$ into two
parts, $C_1$ and $C_2$, both of which are homeomorphic to
$T^2\times I$. Now if $\rho_*[T]=[T]$, then $\rho$ would have to
switch $C_1$ and $C_2$. Since $\rho$ has no fixed points on
$\partial C_1=T\cup\rho(T)$, this would imply that $\rho$ is free, a contradiction.
\end{proof}

 Next, we  show that $L$ must result from $2$-cabling the components of the Hopf link.
\begin{lem}\label{lem:SateHopf}
If $M=\Sigma(L)$, then there exists a genus one Heegaard splitting
$S^3=V_1\cup V_2$, such that $L$ is the union of a closed
$2$--braid in $V_1$ and a closed $2$--braid in $V_2$(see Figure
\ref{fig:hopfcable}).
\end{lem}
\begin{proof}
The last lemma showed that we can assume $T$ and $\rho(T)$ split
$M$ into two parts $C_1$ and $C_2$, with each $C_i$ homeomorphic
to $T^2\times I$. Moreover, since $\rho_*[T]=-[T]$, we have
$\rho(C_i)=C_i$, and $V_i=\pi(C_i)$ is a manifold with torus
boundary. Since $\pi(T)$ is an embedded torus in $S^3$, the two
manifolds $V_1,V_2$ bounded by $\pi(T)$ have the same homology
groups as the solid torus.

Let $h$ be a generator of $H_2(V_1,\partial V_1)\cong \Z$. Since $\pi\co
C_1\to V_1$ is a proper map of nonzero degree, there exists a
primitive element $\tilde h\in H_2(C_1,\partial C_1)$ such that
$\pi_*(\tilde h)=h$. We can choose an annulus $A\subset
C_1=T^2\times I$ representing $\tilde h$, hence $\pi(A)$
represents $h$. By Gabai's theorem that the singular Thurston norm
is equal to the Thurston norm \cite[Corollary~6.18]{Ga1}, the
Thurston norm of $h$ is $0$. Let $G\subset V_1$ be a Thurston norm
minimizing surface in the homology class $h$. Then $\partial G$
represents a primitive element in $H_1(\partial V_1)$. Attaching
annuli to $\partial G$ if necessary, we may assume $|\partial
G|=1$, and hence the component of $G$ containing $\partial G$ is a disk.
Since $V_1\subset S^3$, it follows that $V_1$ is a solid torus.

$C_1$ is homeomorphic to $T^2\times[0,1]$, and $\partial C_1$
consists of two parallel flat tori. The universal cover
$\widetilde{C_1}$ of $C_1$ is a submanifold of $\E^3$ bounded by
two parallel planes. After scaling the metric, $\widetilde{C_1}$
is isometric to $\mathbb E^2\times[0,1]$. Each $\E^2\times t$ is
preserved by the (isometric) action of $\pi_1(C_1)$, let
$R_t=(\E^2\times t)/\pi_1(C_1)$, then $C_1$ is foliated by the
flat tori $R_t$, and the distance between $R_t$ and $R_0$ is $t$.
Since $\rho$ is an isometry and $\rho(R_0)=R_1$, $\rho$ must send
$R_t$ to $R_{1-t}$.

Now as the fixed point set of an isometry, $\pi^{-1}(L)$ is
geodesic. Each $C_i$ must contain some components of
$\pi^{-1}(L)$. Let $K=\pi^{-1}(L)\cap C_1$. Since $K$ is a
geodesic disjoint from the flat tori $R_0$ and $R_1$, $K$ should
be parallel to them. Indeed, since $\rho$ is an isometry, for each
component $K_i$ of $K$
$$\mathrm{dist}(K_i,R_0)=\mathrm{dist}(\rho(K_i),\rho(R_0))=\mathrm{dist}(K_i,R_1),$$
so $K_i$ lies on the torus $R_{\frac12}$. Hence $K\subset
R_{\frac12}$. Let $R_{[0,\frac12]}=\cup_{t\in[0,\frac12]}R_t$,
then $R_{[0,\frac12]}$ is homeomorphic to $T^2\times[0,\frac12]$,
and $\rho(R_{[0,\frac12]})\cap R_{[0,\frac12]}=R_{\frac12}$.
Choose a properly embedded surface $W\subset R_{[0,\frac12]}$ such
that each component of $W$ is an annulus whose boundary consists
of a component of $K$ and a curve on $R_0$. Then $\pi(W)$ is a
disjoint union of annuli in $V_1=\pi(C_1)$, where each annulus
connects a component of $L$ to an essential curve on
$\partial V_1$. It follows that $L\cap V_1$ is a torus link in the
solid torus $V_1$. A simple Euler characteristic count shows that
$L$ intersects each meridian disk of $V_1$ in two points, so
$L\cap V_1$ is a $2$--braid in $V_1$.

The same argument as above shows that $V_2$ is a solid torus, and
$L\cap V_2$ is a $2$--braid in $V_2$. Now $V_1\cup V_2$ is a genus one
Heegaard splitting for $S^3$.
\end{proof}

To complete the theorem, let $L=H_{m,n}\subset S^3$ be the link
from the previous lemma, such that $H_{m,n}\cap V_1$ is isotopic
to the $(2,m)$ torus link in $V_1$ and $H_{m,n}\cap V_2$ is
isotopic to the $(2,n)$ torus link in $V_2$.

\begin{lem}
The manifold $\Sigma(H_{m,n})$ is a torus bundle, and the
monodromy is represented by the matrix
$$\begin{pmatrix}
mn-1&n\\
-m&-1
\end{pmatrix}.
$$
\end{lem}
\begin{proof}
The proof of Lemma~\ref{lem:SateHopf} shows that $\Sigma(H_{m,n})$
is a torus bundle. We only need to determine its monodromy.

Choose two curves $\mu,\lambda$ on the surface $\partial
V_1=-\partial V_2$, such that $\mu,\lambda$ are the meridians of
$V_1,V_2$, respectively. Moreover, the curves are oriented such
that $\mu\cdot\lambda=1$.

Let $T_1,T_2$ be the two components of $\pi^{-1}(\partial V_1)$.
Let
$$\tilde{\mu}_i=\pi^{-1}(\mu)\cap T_i,\qquad\tilde{\lambda}_i=\pi^{-1}(\lambda)\cap T_i.$$

The preimage of the meridian disk of $V_1$ is an annulus which
gives a homology (in $C_1=\pi^{-1}(V_1)$) between
$[\tilde{\mu}_1]$ and $-[\tilde{\mu}_2]$, so
$$[\tilde{\mu}_1]=-[\tilde{\mu}_2]\in H_1(C_1).$$ Moreover, there is a compact
surface $A\subset V_1$, such that $A$ is an annulus when $m$ is odd and $A$
is the union of two annuli when $m$ is even,
and $\partial A$ consists of $H_{m,n}\cap V_1$ and a $(2,m)$ torus link on $\partial V_1$.
The preimage of $A$ gives a homology (in $C_1$) between
$2[\tilde{\lambda}_1]+m[\tilde{\mu}_1]$ and $2[\tilde{\lambda}_2]+m[\tilde{\mu}_2]$. Since
$[\tilde{\mu}_1]=-[\tilde{\mu}_2]$ in $C_1$, it follows that
$$[\tilde{\lambda}_1]=[\tilde{\lambda}_2]+m[\tilde{\mu}_2]\in
H_1(C_1).$$

In $V_2$, the roles of $\lambda$ and $\mu$ are switched. The same
argument as above shows that
$$[\tilde{\lambda}_1]=-[\tilde{\lambda}_2],\qquad [\tilde{\mu}_1]=[\tilde{\mu}_2]+n[\tilde{\lambda}_2]$$
in $H_1(C_2)$.

Consider the manifold $C_1\cup_{T_2}C_2$, which is homeomorphic to
$T^2\times I$. There are two copies of
$\tilde{\mu}_1,\tilde{\lambda}_1$ on its boundary, and their
homological relation can be computed as follows:
\begin{eqnarray*}
[\tilde{\mu}_1]&=&[\tilde{\mu}_2]+n[\tilde{\lambda}_2]\\
&=&-[\tilde{\mu}_1]+n([\tilde{\lambda}_1]+m[\tilde{\mu}_1])\\
&=&(mn-1)[\tilde{\mu}_1]+n[\tilde{\lambda}_1]
\end{eqnarray*}
\begin{eqnarray*}
[\tilde{\lambda}_1]&=&-[\tilde{\lambda}_2]\\
&=&-([\tilde{\lambda}_1]+m[\tilde{\mu}_1]).
\end{eqnarray*}
It follows that the monodromy of the torus bundle is given by the
matrix $\begin{pmatrix}
mn-1&n\\
-m&-1
\end{pmatrix}$.
\end{proof}

We are now able to  finish the proof of Proposition ~\ref{prop:TreDoub}.  If $\Sigma(H_{m,n})$ is homeomorphic to the zero surgery on the
trefoil knot, then the monodromy of the torus bundle has
order $6$. This implies that the trace of
the matrix is $1$. By the preceding lemma we find that $mn=3$, and hence $m=\pm1,n=\pm3$
or $m=\pm3,n=\pm1$. Since $H_{n,m}\simeq H_{m,n}$, the link is
$H_{1,3}$ or $H_{-1,-3}$.
\end{proof}

\begin{rem}
Note that our proof applies equally well to classify links whose
double branched cover gives rise to any torus bundle possessing a
Euclidean structure.  In particular, we recover the well-known
result that $T^3$ is not a double branched cover of a link in
$S^3$ \cite{HirNeum}.
\end{rem}

\section{Links with rank $4$ Khovanov homology}

In this section, we prove Theorem \ref{thm:KhRank4}. The theorem
will follow quickly from the results of the preceding two
sections, together with Proposition~\ref{prop:ssbound}. It will be
useful, however, to first understand essential spheres in branched
double covers. The following proposition is a well-known
consequence of the Equivariant Sphere Theorem, originally proved in
\cite{MSY}. (See also \cite{Du} for the version we use here.)

\begin{prop}\label{prop:CovIrr}
(1) Suppose $L$ is a link in $S^3$. If $L=L_1\#L_2$, then
$\Sigma(L)=\Sigma(L_1)\#\Sigma(L_2)$. If $L=L_1\sqcup L_2$, then
$\Sigma(L)=S^1\times S^2\#\Sigma(L_1)\#\Sigma(L_2)$.

(2) Suppose $L$ is a non-split prime link in $S^3$, then
$\Sigma(L)$ is irreducible.

(3) If $L$ is a non-split link, then $\Sigma(L)$ contains no
$S^1\times S^2$ summand.
\end{prop}
\begin{proof}
\vspace{5pt}(1) This fact is obvious.

\vspace{5pt}(2) Assume $\Sigma(L)$ is reducible. By the
Equivariant Sphere Theorem, there exists an essential sphere
$S\subset\Sigma(L)$, such that $\rho(S)=S$ or $\rho(S)\cap
S=\emptyset$.

If $\rho(S)=S$, then $S$ doubly branched covers $\pi(S)$, hence
$\pi(S)$ is an embedded surface in $S^3$, and it is either a disk,
a sphere or a projective plane. The last case is immediately ruled
out since $S^3$ does not contain any embedded projective plane.

If $\pi(S)$ is a disk, then $\partial(\pi(S))$ is a component of
$L$ which bounds a disk in the complement of $L$. This
contradicts the assumption that $L$ is non-split.

If $\pi(S)$ is a sphere, then $\pi|_S$ is a $2$--fold branched
covering with two ramification points. So $L$ intersects $\pi(S)$
in exactly two points. The sphere $\pi(S)$ splits $S^3$ into two
balls $B_1,B_2$. If $L\cap B_i$ is a trivial arc in $B_i$, then
$\pi^{-1}(B_i)$ is a ball bounded by $S$ in $\Sigma(L)$,
contradicting the assumption that $S$ is essential. So $L$ is a
nontrivial connected sum, which is impossible since $L$ is prime.

If $\rho(S)\cap S=\emptyset$, then $\pi(S)$ is an embedded sphere
in the complement of $L$. Since $L$ is non-split, $\pi(S)$ bounds
a ball in $S^3-L$. It follows that $S$ bounds a ball in
$\Sigma(L)$, a contradiction.

\vspace{5pt}(3) If $L$ is non-split, there exists a collection of
spheres $S_1,\dots,S_n\subset S^3$, such that each $S_i$
intersects $L$ in exactly two points, and they decompose $L$ as a
connected sum of non-split prime links $L_1,\dots,L_{n+1}$. Thus
$\Sigma(L)$ is a connected sum of $\Sigma(L_i)$'s. By (2), each
$\Sigma(L_i)$ is irreducible, so $\Sigma(L)$ contains no
$S^1\times S^2$ summand.
\end{proof}

\begin{proof}[Proof of Theorem~\ref{thm:KhRank4}] Let $L$ be a link with $\det(L)=0$.   By Theorem \ref{thm:Lee}, if $L$ is an $n$--component link, then
$\mathrm{rank}\:Kh(L)\ge 2^n$. So we can assume $L$ has two
components (it cannot have one-component, since knots satisfy $\det(K)\ne0$).

If $L$ is split with two components $K_1,K_2$, then
$$\mathrm{rank}_{\mathbb F}Kh(L;\mathbb F)=
\mathrm{rank}_{\mathbb F}Kh(K_1;\mathbb
F)\times\mathrm{rank}_{\mathbb F}Kh(K_2;\mathbb F),$$ where $\mathbb F=\Z/2\Z$.
 Indeed,
there are chain complexes for which $CKh(L)\cong CKh(K_1)\otimes
CKh(K_2)$. Thus $L$ has Khovanov rank $4$ (over $\mathbb F$) if
and only if each $K_i$ has Khovanov rank $2$.

If $L$ is non-split, then $\Sigma(L)$ contains no $S^1\times S^2$
summand by Proposition~\ref{prop:CovIrr}. Since $\det(L)=0$,
$b_1(\Sigma(L))>0$ (recall that $|\det(L)|=|H_1(\Sigma(L);\Z)|$ when $\det(L)\ne0$ and, if $\det(L)=0$, that $H_1(\Sigma(L);\Z)$ has infinite order).
Moreover, the fact that $L$ has two components implies that
$b_1(\Sigma(L))\le 1$, and hence $b_1(\Sigma(L))=1$.

Since $\mathrm{rank}_\F Kh(L;\F)=4$, Proposition~\ref{prop:ssbound}
implies that  $\mathrm{rank}\:\HFa(\Sigma(L))\le 2$. For a
manifold with $b_1(M)=1$, \cite[Theorem~10.1]{OSzAnn2} and Lemma
$2.3$ imply that $\mathrm{rank}\:\HFa(M)\ge 2$, so we see that
$\mathrm{rank}\:\HFa(\Sigma(L))=2$.   Calling on
Corollary~\ref{cor:Rank2Gen}, we see that $\Sigma(L)=\pm
S^3_0(3_1)\#Z$ for some homology sphere, $Z$. Now by
Proposition~\ref{prop:CovIrr} and the uniqueness of the
Kneser--Milnor prime decomposition, $L$ has a connected summand,
$L_0$, such that $\Sigma(L_0)=\pm S^3_0(3_1)$.
Proposition~\ref{prop:TreDoub} shows that, up to taking mirror images,
$L_0$ is either $H_{\sd1,\sd 3}$ or $H_{1,3}$. Since $L_0$ has two components, the 
other prime summands of $L$ are all knots. 

Now the connected sum formula for the Jones polynomial
$$J(L'\#L'')=J(L')J(L'')$$
implies that the Jones polynomial $J_L$
of $L$ is the product of $J_{L_0}$ with a nonzero polynomial.

Consider the case $L_0=H_{\sd 1,\sd 3}$.  The Jones polynomial of $H_{\sd 1,\sd 3}$ is $$-t^{-\frac{23}2}+t^{-\frac{21}2}-t^{-\frac{13}2}-t^{-\frac92},$$
so $J_L$ is nonzero. Since the Euler characteristic of $\widetilde{Kh}(L)$ is $J_L$, and by assumption $\mathrm{rank}_\F Kh(L;\F)=2\cm\mathrm{rank}_\F \widetilde{Kh}(L;\F)=4$, it follows that $J_L$ has exactly two terms i.e. is of the form $\pm t^{\frac{p}{2}}\mp t^{\frac{q}{2}}$.
This implies that all nonzero roots of $J_{L_0}$ are roots of unity. We claim that this is impossible.

Indeed, if we multiply $J_{H_{\sd 1,\sd 3}}$ by $-t^{\frac{23}2}$, we get
$$f(t)=t^7+t^5-t+1.$$
Assume $r=e^{i\theta}$ is a root of $f(t)$, then $r^7+r^5$ is a nonzero number with argument $6\theta$,
and $-r+1$ is a nonzero number with argument $\frac{\theta-\pi}2$. It follows that
\begin{eqnarray*}
6\theta&=&\frac{\theta-\pi}2+(2k+1)\pi\\
\theta&=&\frac{4k+1}{11}\pi.
\end{eqnarray*}
Hence $r$ is a root of $t^{11}+1$. Since $\frac{t^{11}+1}{t+1}$ is an irreducible polynomial of degree $10>7$,
the only root of unity which is also a root of $f(t)$ is $-1$.  This contradiction implies that $L$ must be split.

The case that $L_0=H_{1,3}$ follows similarly.
\end{proof}

\begin{proof}[Proof of Corollary~\ref{cor:Khcor}]
Assume that Khovanov homology detects the unknot. If $Kh(L)\cong
Kh(U_2)$, where $U_2$ is the two-component unlink, then
$\mathrm{rank}\: Kh(L)=4$.  Also, the Jones polynomial, $J_L(q)$,
is equal to that of the unlink.  In particular, we have
$|\det(L)|=|J_L(-1)|=0$. Theorem \ref{thm:KhRank4} now shows that $L$
is isotopic to a split link, each component of which has the
Khovanov homology of the unknot.   Our assumption implies that
both components are unknotted.

Conversely, if Khovanov homology detects the two-component unlink,
then it clearly detects the unknot (given a knot $K$, consider the
split link obtained from $K$ union an unknot).
\end{proof}

\end{document}